\documentclass[11pt]{elsarticle}

\usepackage{amsfonts}
\usepackage{amssymb}
\usepackage{graphicx,color}
\usepackage{epstopdf}
\usepackage{epsfig}

\usepackage{amsmath}

\usepackage{amsthm}
\allowdisplaybreaks

\newcommand{\field}[1]{\mathbb{#1}}
\newcommand{\R}{\field{R}}
\newcommand{\be}{\begin{equation}}
	\newcommand{\ee}{\end{equation}}

\newtheorem{theorem}{Theorem}
\newtheorem{lemma}[theorem]{Lemma}
\newtheorem{corollary}[theorem]{Corollary}
\newtheorem{proposition}[theorem]{Proposition}

\theoremstyle{remark}
\newtheorem{remark}[theorem]{Remark}

\theoremstyle{definition}

\begin{document}
	\begin{frontmatter}
		\title{On the Best Uniform Polynomial Approximation to the Checkmark Function }
		
		\author[FortWayne]{P. D. Dragnev \fnref{Dragnev}}
		\fntext[Dragnev]{The research of this author was supported in part by NSF grant DMS-1936543.}
		\ead{dragnevp@pfw.edu}
		
		\author[FortWayne]{A. R. Legg \corref{Legg}}
		\cortext[Legg]{Corresponding Author}
		\ead{leggar01@pfw.edu}
		
		\author[ULL]{R. Orive \fnref{Orive}}
		\fntext[Orive]{The research of this author was supported in part by Ministerio de Ciencia e Innovaci\'{o}n under grant MTM2015-71352-P,
			and the PFW Scholar-in-Residence program.}
		\ead{rorive@ull.es}
		
		\address[FortWayne]{Department of Mathematical Sciences, Purdue University Fort Wayne, Ft. Wayne, IN 46805}
		\address[ULL]{Departamento de An\'{a}lisis Matem\'{a}tico, Universidad de La Laguna, 38200, The Canary Islands, Spain }
		
		\date{\today}

		\begin{abstract}
			The best uniform polynomial approximation of the checkmark function $f(x)=|x-\alpha |$ is considered, as $\alpha$ varies in $(-1,1)$. For each fixed degree $n$, the minimax error $E_n (\alpha)$ is shown to be piecewise analytic in $\alpha$. In addition, $E_n(\alpha)$ is shown to feature $n-1$ piecewise linear decreasing/increasing sections, called V-shapes. The points of the alternation set are proven to be piecewise analytic and monotone increasing in $\alpha$ and their dynamics are completely characterized. We also prove a conjecture of Shekhtman that for odd $n$, $E_n(\alpha)$ has a local maximum at $\alpha=0$.
		\end{abstract}
		
	\end{frontmatter}
	
	\setcounter{equation}{0}
	\setcounter{theorem}{0}
	\section{Introduction}
	\label{intro}
	
	Our purpose is to study the best polynomial approximation (in the uniform norm) to the so--called checkmark function, i.e.
	\begin{equation}\label{Checkmarkfunction}
		f(x) = f(x;\alpha) = |x-\alpha|, \quad x\in [-1,1], \, \alpha \in (-1,1).
	\end{equation}
	
	Given a nonnegative integer $n$ and $\alpha \in (-1,1)$ we denote by $p_n(x) = p_n(x;\alpha)$ the best polynomial approximation of degree at most $n$ to the function $f(x)$ in \eqref{Checkmarkfunction}. This $p_n(x; \alpha)$ is known as the minimax polynomial and satisfies the condition
	\begin{equation}\label{defminimax}
		E_n(\alpha) := \|f - p_n\| = \min_{q\in \mathbb{P}_n}\,\|f - q\|\,,
	\end{equation}
	where $\mathbb{P}_n$ denotes the linear space of all polynomials of degree at most $ n$ and $\|\cdot\|$ denotes the uniform (or Chebyshev) norm in the interval $[-1,1]$.
	
	The study of the best polynomial or rational approximation to the particular case $f(x;0) = |x|$ is a classical problem in Approximation Theory (see Bernstein \cite{Bernstein1, Bernstein2}, Nikolsky \cite{N}, Petrushev and Popov \cite{PePo}, Stahl \cite{Stahl}, and Totik \cite{T} among many others), since it is one of the simplest nontrivial functions to be approximated. The case of $f(x;\alpha)$ is treated by Bernstein \cite{Bernstein2}. In these classical treatments of the problem, the focus is on asymptotic rates of approximation when the degrees of the polynomials (or rational functions) tend to infinity. From these studies it is known that as $n \to \infty$, the value $nE_n(\alpha) \to \sigma \sqrt{1-\alpha^2}$, where  $\sigma:=\lim_{n\to \infty} nE_n(0) \approx 0.28\dots$ is the Bernstein constant (see \cite{Bernstein1}).

	In contrast, our objective here is to consider the evolution of the solution to the minimax problem for fixed degree $n$ as $\alpha$ varies from $-1$ to $+1$. We will be especially concerned with the smoothness and behavior of $E_n(\alpha)$, and with the positions and phase transitions of the so-called `Chebyshev alternation points,' whose definition we now review.
	
	It is well known that the best polynomial approximation in the uniform norm to $f$ of degree at most $n$ is uniquely characterized by the following equioscillation property. There exist at least $n+2$ points $z_1, z_2, \cdots z_N$ in $[-1,1]$ such that the minimax error $E_n(\alpha)$ defined in \eqref{defminimax} is attained with alternating signs, that is,
	\begin{equation}\label{equioscil}
		e_n(z_i;\alpha) := f(z_i;\alpha) - p_n(z_i;\alpha) = \epsilon\,(-1)^i\,\|f - p_n\| = \epsilon\,(-1)^i\,E_n(\alpha)\,,
	\end{equation} 
	for $i=1,2, \cdots, N$, where $\epsilon = \pm 1\,.$ The points $z_i$ are called alternation points. 
	This characterization is the basis for the numerical algorithm to compute the minimax polynomial, which is commonly known as the Remez algorithm (see \cite{Remez} or \cite[Ch. 1]{PePo}).

	In \cite[Lemma 3]{DLT} it is proven that for the checkmark function the exact number of alternation points is either $n+2$ or $n+3$, with $\alpha$ being always an alternation point such that 
	\begin{equation}\label{p=E}
		p_n(\alpha;\alpha) = E_n(\alpha)\,.
	\end{equation}
	
	Hereafter, we denote by $u_i =u_i (\alpha)$ the alternation points located to the left of $\alpha$, and by $v_j=v_j (\alpha)$, those to the right. The $u_i$ and $v_j$ are enumerated by their distance to $\alpha$.  So, we write
	\begin{equation}\label{mathcal A}
		{\mathcal A}_n(\alpha) := \{u_{k+1} (\alpha)< \dots <u_1(\alpha) <  \alpha < v_1 (\alpha) < \dots < v_{\ell+1}(\alpha) \}\,,
	\end{equation} 
	where $k,\ell \geq 0$. In particular, for $\alpha$ not in a so-called ``V-shape" (see next paragraph) we know from \cite{DLT} that the alternation set $\mathcal{A}_n(\alpha)$ contains exactly $n+2$ alternation points, and $\{-1,\alpha, 1\} \subset \mathcal{A}_n(\alpha)$. In this case $u_{k+1}(\alpha)=-1$, $v_{\ell+1}(\alpha)=1$, and $k+\ell = n-1$. For convenience we define $u_0(\alpha):=\alpha=:v_0(\alpha)$.

	An interesting aspect of the error function $E_n(\alpha)$ as $\alpha$ varies from $-1$ to $1$ is that it develops what we call {\it V-shapes}. We say $E_n(\alpha)$ has a V-shape on the domain interval $[a,b]$ if there is a $c \in (a,b)$ such that $E_n(\alpha)$ is linear and decreasing on $[a,c]$, linear and increasing on $[c,b]$, and continuous on $[a,b].$ If $[a,b]$ is a {\it maximal} such interval, then we call the portion of the graph of $E_n(\alpha)$ over $[a,b]$ a V-shape, and we call the points $(a,E_n(a)),(b, E_n(b))$ in the graph of $E_n(\alpha)$ the {\it endpoints of the V-shape}. The point $(c,E_n(c))$ we call the {\it tip of the V-shape}. By a slight abuse of terminology we will also use the term {\it V-shape} in reference to the domain alone (instead of the graph). For instance, we might say $[a,b]$ is a V-shape, whose endpoints are $a,b$, and whose tip is $c$.
	
	  An important result previously established in \cite[Theorem 2]{DLT} asserts that for values of $\alpha$ close to $+1$ (or similarly to $-1$, by the symmetry of the problem), the error function $E_n(\alpha)$ can be described in terms of the classical Chebyshev polynomials. And if $(\alpha_3,1]$ is the largest interval on which this description by Chebyshev polynomials holds, then on a contiguous interval of values $\alpha \in [\alpha_1,\alpha_3]$, the graph of $E_n(\alpha)$ has a V-shape with tip $\alpha_2 \in (\alpha_1, \alpha_3).$ The proof is based on the fact that for such range of values of $\alpha$, the domain endpoint $x=+1$ is a soft--endpoint, that is, the minimax polynomial for $f$ in $[-1,1]$ is also the minimax polynomial for intervals of the form $[-1, b]$, for $b \geq b_0$, with $b_0 < 1$. This allows us to perform a simple linear transformation of the problem from which the existence of the V-shape arises. As noted in a Remark on p. 154 of \cite{DLT} when the alternation set has $n+3$ points for some $\alpha$ then this $\alpha$ is a tip of a V-shape for $E_n(\alpha)$.
	In fact, on p. 151 of \cite{DLT} it is conjectured  that the number of such V-shapes is exactly $n-1$. Our Theorem \ref{thm:vshapes} below establishes this conjecture.
	
	The simplest nontrivial cases are those of $n=2$ and $n=3$, which were studied in \cite{DLT}. It is illustrative to review these.  It is easy to plot $E_2(\alpha)$ and $E_3(\alpha)$ numerically as functions of $\alpha$, yielding Figure \ref{fig:2case}.
	\begin{figure}[h]
		\begin{center}
			\includegraphics[scale=0.27]{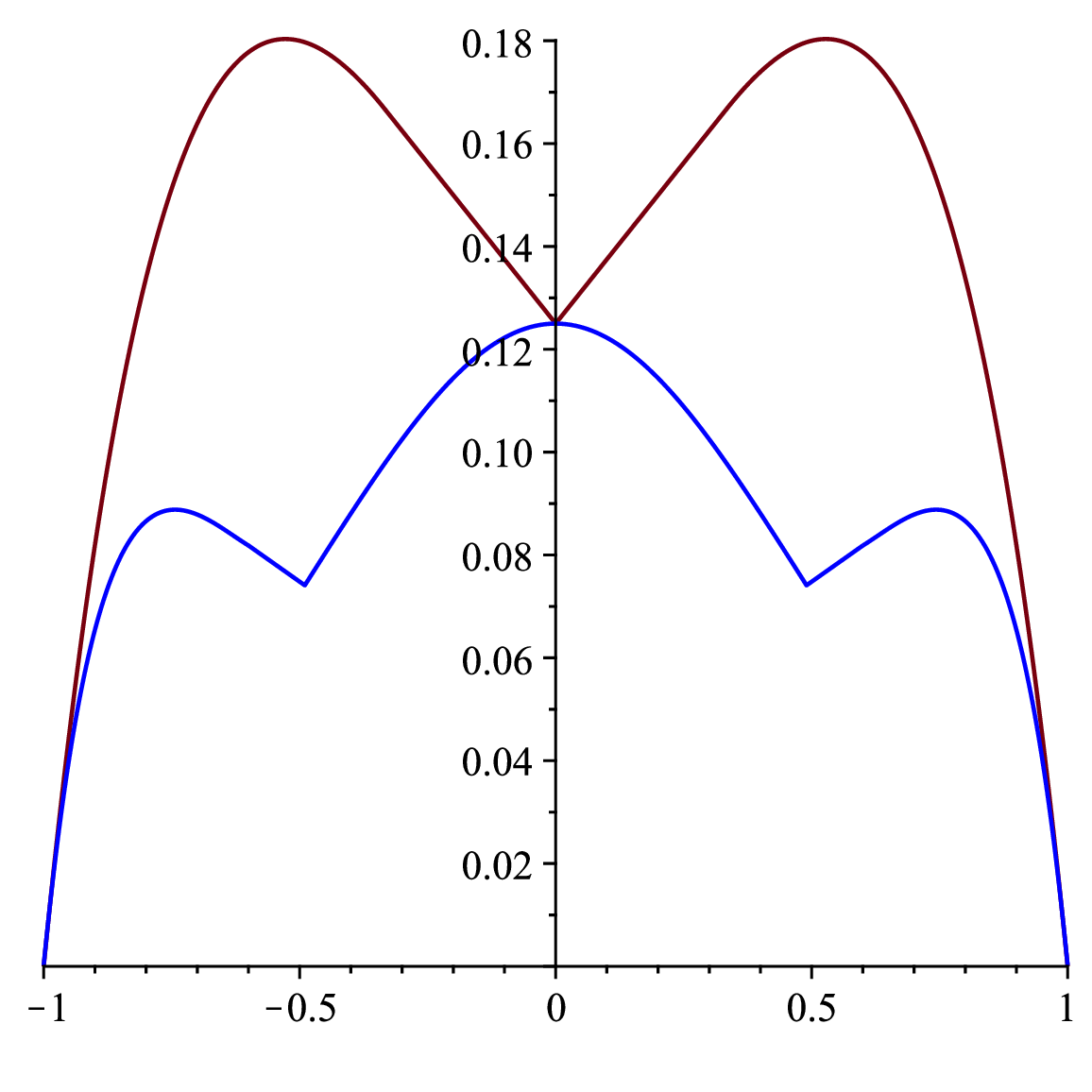}
		\end{center}
		\caption{The graphs of $E_2(\alpha)$ and $E_3(\alpha)$. }
		\label{fig:2case}
	\end{figure}
	Observe that the graph of $E_2(\alpha)$ has a V-shape in a symmetric interval about $\alpha=0$, where the ``tip'' of the V-shape is located. Outside the V-shape, $E_2(\alpha)$ can be determined using Chebyshev polynomials as mentioned above. The graph of $E_3(\alpha)$ is slightly more involved, showing two symmetric V-shapes, and now $\alpha=0$ is a local and in fact absolute maximum. At $\alpha=0$, note that the two graphs intersect, reflecting the fact that $p_2(x;0)=p_3(x;0)$. 
	
	In Figure \ref{fig:decreasing}, we numerically plot $E_n(\alpha)$ for $n=1,2, \cdots, 7.$ As $n$ increases, it appears that more and more V-shapes arise at whose tips the graphs of $E_n(\alpha)$ and $E_{n+1}(\alpha)$ coincide.  To illustrate the asymptotic rate of approximation, in Figure \ref{fig:asymp} we also show the plots of $n E_n(\alpha)$ for various values of $n$, noting that they tend to accumulate along $\sigma\sqrt{1-\alpha^2}$.
	\begin{figure}[h]
		\begin{center}
			\includegraphics[scale=.27]{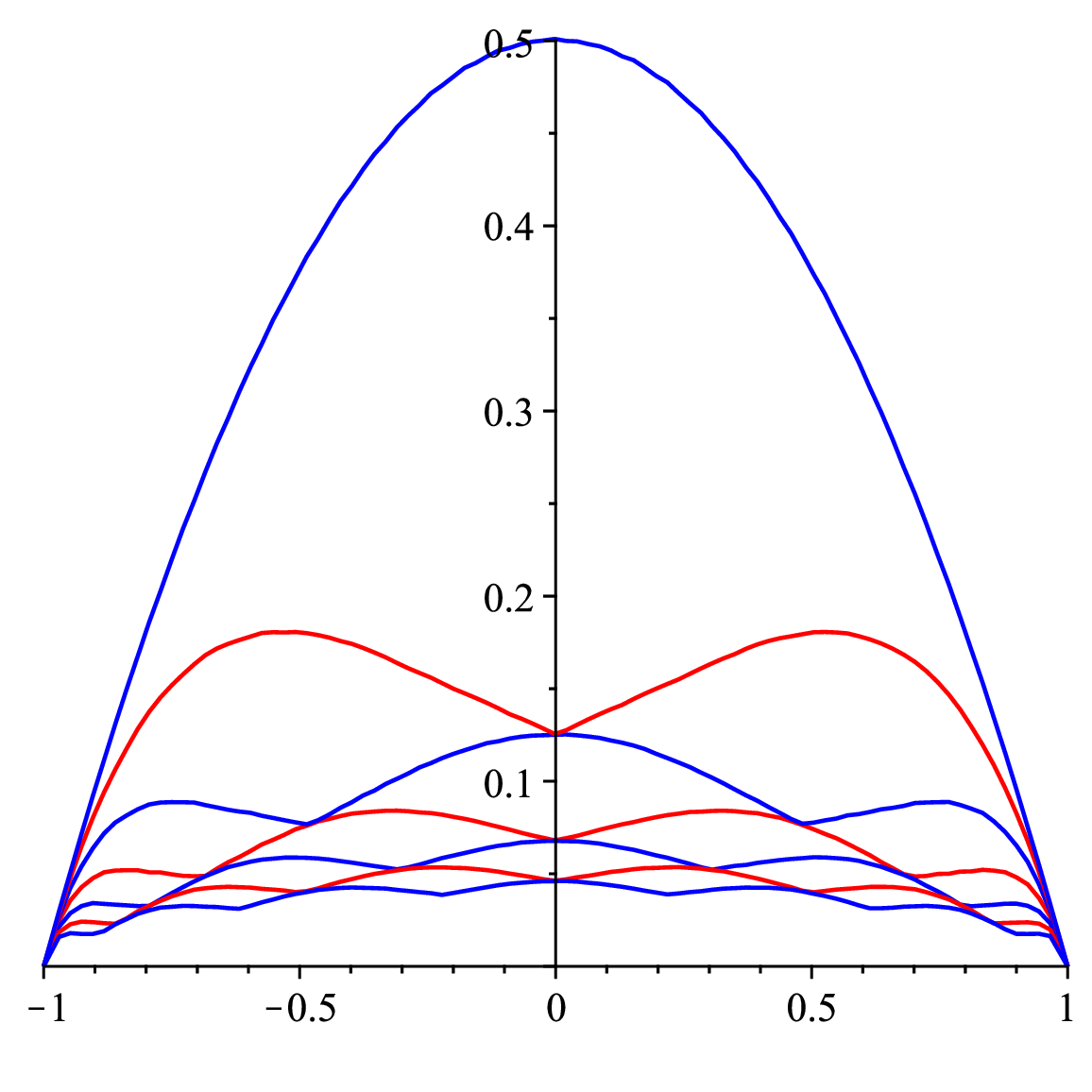}
		\end{center}
		\caption{The graphs of $E_n(\alpha)$, for $n=1,2,\ldots,7$. }
		\label{fig:decreasing}
	\end{figure}
	\begin{figure}
		\begin{center}
			\includegraphics[scale=.27]{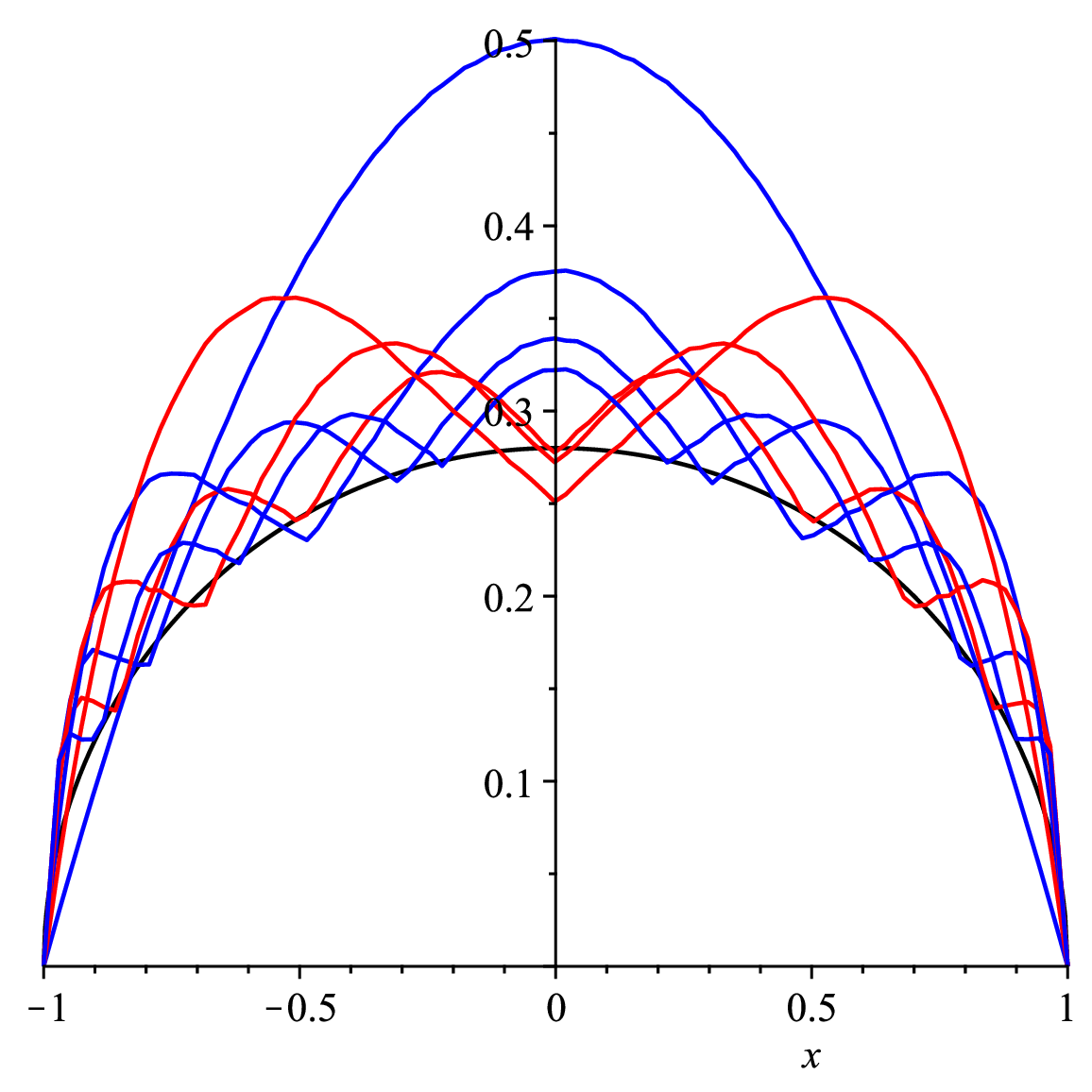}
		\end{center}
		\caption{The normalized graphs $nE_n(\alpha)$, for $n=1,2,\ldots,7$ and the limit $\sigma\sqrt{1-\alpha^2}$. }
		\label{fig:asymp}
	\end{figure}
	
	As we said, our main concern is the study of $E_n(\alpha)$ as $\alpha$ varies in $(-1,1)$, for a fixed nonnegative integer $n$. In the next section, the monotonicity of the alternation points as functions of the parameter $\alpha$ is established. In Section $3$ we show piecewise analyticity of the minimax error $E_n(\alpha)$ {and the alternation nodes}; the points where analyticity fails are the tips and endpoints of V-shapes.  We also prove a conjecture of Shekhtman \cite{Shekt} that $E_n(\alpha)$ has a local maximum at $\alpha=0$ when $n$ is odd. In Section 4, the conjecture about the exact number of V-shapes in the graph of $E_n(\alpha)$, posed in \cite{DLT}, is resolved, which also allows us to describe the phase transitions of the alternation points. These phase transitions are illustrated in Figure \ref{fig:phase} for $n=5$. 
	
	\setcounter{equation}{0}
	\setcounter{theorem}{0}
	\section{Monotonicity of the alternation points}

	We start our analysis by discovering monotonicity of the alternation points with respect to the parameter $\alpha$, where $\alpha$ is not a tip of a V-shape. First note the following fact about continuity.
	
	\begin{lemma}\label{lem:continuity}
		Given a nonnegative integer $n$, the minimax polynomial $p_n(x;\alpha)$ and the minimax error $E_n(\alpha)$ vary continuously in the parameter $\alpha \in [-1,1].$
	\end{lemma}
	The proof follows from the continuity of the best approximation operator for subspaces of finite dimension  (see e.g. \cite[Th. 1.2, p. 60]{deVore}).

	\begin{lemma}\label{MVL}
		Suppose $f(x)$ is a nonconstant continuous function on $[a,b]$ and differentiable on $(a,b)$ and that $f(a)\geq 0 \geq f(b)$ {\rm (}resp. $f(a)\leq 0 \leq f(b)${\rm)}. Then there is a point $z \in (a,b)$ such that $f' (z)<0$  {\rm (}resp. $f' (z)>0${\rm)}.
	\end{lemma}
	
	\begin{proof} This is a consequence of the Mean Value Theorem.
	\end{proof}
	
	In what follows, it will be useful to have a notation for a normalized error function. For each $\alpha \in (-1,1)$, let us define the function
	\begin{equation}\label{g alpha}
		g_{\alpha}(x) := \,\frac{p_n(x;\alpha) - |x - \alpha|}{E_n(\alpha)}\,.
	\end{equation}
	Thus, $g_{\alpha}$ is normalized in such a way that $|g_{\alpha}(x)|\leq 1\,,\,x\in [-1,1]\,.$ 
	
	\begin{theorem}\label{thm:monotonicity}
		On any interval of $\alpha$ values not including a tip of a V-shape, the alternation points $\{u_i(\alpha)\}_{i=0}^{k}$ and  $\{v_j(\alpha)\}_{j=0}^{\ell}$ are non-decreasing in $\alpha$.
	\end{theorem}
	\begin{proof}
		For $\alpha$ in a linear part of a V-shape the monotonicity follows from a linear transformation (see \cite[Remark on p. 154]{DLT}). Assume henceforth that $\alpha$ is not in a V-shape, a case described immediately after \eqref{mathcal A} (recall that then $u_{k+1} =-1$ and $v_{\ell+1} = +1$). We shall prove first that the alternation points in question are non-decreasing.
		
		{ It is clear that $g_\alpha'(u_i (\alpha))=0$, $i=1,\dots,k$ and $g_\alpha'(v_j(\alpha))=0$, $j=1,...,\ell$. By Rolle's theorem this accounts for a root of $g_\alpha''(x)$ in each of the $n-3$ subintervals $(u_{i+1}(\alpha),u_i(\alpha))$, $i=1,\dots, k-1$, and  $(v_j(\alpha),v_{j+1} (\alpha))$, $j=1,\dots, \ell-1$. This accounts for $n-3$ zeros of $g''_\alpha (x)$ ($n-2$ zeros if $k$ or $\ell$ is zero). Should $g_\alpha''$ vanish at any of the local extrema $u_i (\alpha)$ or $v_j (\alpha)$, then $g_\alpha'''$ will also vanish there, which yields that $g_\alpha''(x)$ has more than $n-2$ zeros, counting multiplicity, a contradiction. Therefore,  $g''_\alpha(u_1(\alpha))>0$ and $g''_\alpha(v_1(\alpha))>0$ (see \eqref{p=E}).
			Should $g''_\alpha (x)<0$ anywhere on $(u_1 (\alpha),v_1 (\alpha))$, we will obtain two additional zeros of $g_\alpha ''$, which yields a contradiction. Therefore, $g'_\alpha(\alpha^-)>0$ (otherwise, the mean value theorem will imply that $g''_\alpha (\xi)<0$ for some $\xi \in (u_1 (\alpha),\alpha)$).
		}

		Now, let us take $\beta > \alpha$ sufficiently close to $\alpha$ { such that $g'_\beta(\alpha)>0$ and}
		$${\mathcal A}_n(\beta) = \{-1< u_k(\beta) < \ldots < u_1(\beta) < \beta < v_1(\beta) < \ldots < v_\ell (\beta) < 1\}\, .$$
		Consider the function
		\begin{equation}\label{functiong}
			G(x) = G(x;\alpha,\beta) := \,g_{\alpha}(x)\,-\,g_{\beta}(x)\,,x\in \R\,,
		\end{equation}
		with $g_{\alpha}$ and $g_{\beta}$ given by \eqref{g alpha}.  As an illustration, Figure \ref{fig:6case} depicts the case $n=6$ and the graphs of $g_\alpha$, $g_\beta$, and $G$ for $\alpha=0.4$ and $\beta = 0.43$.
		
		\begin{figure}[h]
			\begin{center}
				\includegraphics[scale=0.5]{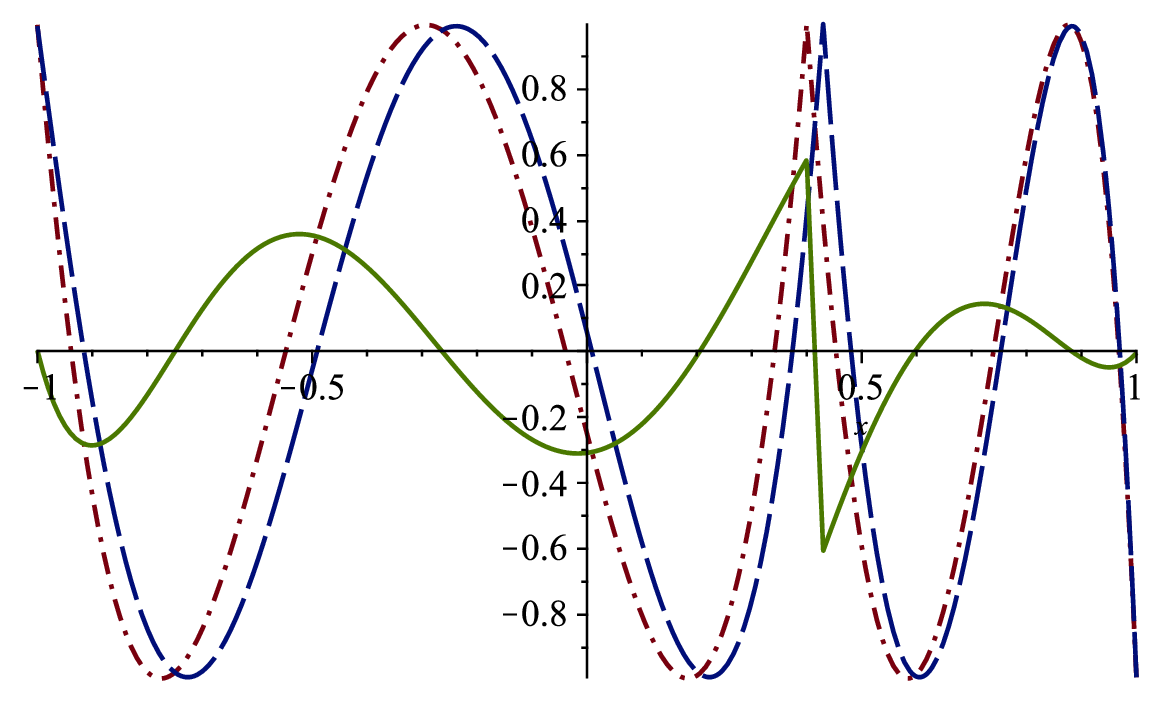}
			\end{center}
			\caption{Graphs of functions $g_\alpha$ (dash-dot), $g_\beta$ (dash) and $G$ (solid) for $n=6, \alpha=0.4$, and $\beta=0.43$. }
			\label{fig:6case}
		\end{figure}

		We note that $G$ is a continuous and piecewise polynomial function differentiable on $\mathbb{R}\setminus \{\alpha, \beta\}$, whose second derivative $G''(x)$ is a polynomial of degree at most $n-2$ with removable singularities at $\alpha$ and $\beta$. By continuity, for $\beta$ close to $\alpha$ we may assume that $u_i(\beta)$ is close to $u_i(\alpha)$ and $v_j(\beta)$ is close to $v_j(\alpha)$, so that $sign\,e_n(u_i(\alpha);\alpha) =  sign\,e_n(u_i(\beta);\beta)\,,i=1,\ldots,k\,,$ as well as $sign\,e_n(v_j (\alpha);\alpha)=sign\,e_n(v_j (\beta);\beta)\,,j=1,\ldots,\ell$, where $e_n(\, \cdot \, ; \, \cdot \,)$ is given by (\ref{equioscil}). Observe that $G(-1) =G(1)= 0$.

		We will account for all the zeroes of $G''(x)$ and show that they are simple and located in $(-1,1)$. Note first that $ G(u_0(\alpha))=G(\alpha)>0$. { Indeed, by \eqref{p=E} and \eqref{g alpha}, we have that $g_\beta (x)$ is a polynomial on the interval $[-1,\beta]$ with $g_\beta (x)\leq 1$. Since we chose $\alpha$ and $\beta$ close enough to have $g'_\beta(\alpha)>0$ and $g_\beta (\alpha)\leq 1$, we guarantee that $g_\beta (\alpha) <1$. }
		
		Since $ G(u_1(\alpha)) \leq 0$, by Lemma \ref{MVL} we have a point $z_0\in (u_1(\alpha),u_0(\alpha))$ at which $G'(z_0)>0$. Similarly, because $ G(u_2(\alpha)) \geq 0$ and $ G(u_1(\alpha)) \leq 0$, there is a point $z_1\in (u_2(\alpha),u_1(\alpha))$ at which $G'(z_1)<0$. { (We remark that $G(x)$ cannot be a constant on $(u_2(\alpha),u_1(\alpha))$, because if it were, it would have to be zero and the equality would hold on $[-1,\alpha]$, but we already established that $G(\alpha)>0$.)} Therefore, there is $y_1 \in (z_1,z_0)\subset (u_2(\alpha),u_0(\alpha))$ such that $G''(y_1)>0$.
		
		Continuing the argument in a similar fashion we derive the existence of points $z_i\in(u_{i+1}(\alpha), u_{i}(\alpha))$, $i=0,\dots,k$, such that $sign\, G'(z_i)=(-1)^{i}$, which yields the existence of points $y_i\in(z_i,z_{i-1})\subset (u_{i+1}(\alpha), u_{i-1}(\alpha))$, $i=1,\dots,k$, such that $sign\, G''(y_i)=(-1)^{i-1}$. Observe that $y_i$ and $y_{i+2}$ belong to distinct intervals and $G''(y_i)$ and $G''(y_{i+1})$ have opposite signs, hence all of the $y_i$'s are distinct. Therefore, by the Intermediate Value Theorem there are at least $k-1$ zeros of $G''(x)$ in the interval $(y_k,y_1)$.
		We note that if $\ell=0$, we have accounted for all $(n-2)$ zeros of $G''(x)$.
		
		Analogously, we derive the existence of $\zeta_j \in (v_{j}(\beta), v_{j+1}(\beta))$, $j=0,\dots,\ell$, such that $sign\, G'(\zeta_j)=(-1)^{j}$ and points $\eta_j \in (v_{j-1}(\beta), v_{j+1}(\beta))$, $j=1,\dots,\ell$, such that $sign\, G''(\eta_j)=(-1)^{j}$. This adds another $\ell-1$ zeros of $G''(x)$ in the interval $(\eta_1,\eta_\ell)$. Finally, since $G''(y_1)>0>G''(\eta_1)$ we account for one more zero of $G''(x)$ in $(y_1,\eta_1)$. In summary, all the zeroes of $G''(x)$ are accounted for and belong to the interval $(y_k,\eta_\ell)\subset (-1,1)$. 
		
		{ To prove that the alternation points are non-decreasing it suffices by symmetry to consider only the $u_i$, $i=1,\dots,k$. Suppose to the contrary that there is a (smallest) index $i$, such that $u_i (\beta) < u_i (\alpha)$. Then $sign \, G(u_i(\alpha)) =(-1)^i$ and $ sign \,  G(u_i(\beta))=(-1)^{i+1}$. This implies there is a point $\widehat{z_{i}} \in (u_i(\beta), u_i(\alpha))$ such that $sign \, G'(\widehat{z_{i}})=(-1)^i$. Next, we modify the construction above to derive the existence of $\widehat{z_j} \in (u_{j}(\beta),u_{j-1}(\beta))$, $j=i+1,\dots,k+1$, such that $sign \,G' (\widehat{z_j})=(-1)^j $. Utilizing the alternating sign changes of $G'(x)$ on the set $\{\widehat{z_{k+1}},\dots,\widehat{z_i},z_{i-1},\dots,z_0 \}$ we obtain at least $n-1$ zeros of $G''(x)$, which is a contradiction.}
		
	\end{proof}

	\begin{remark}
		\label{nosignchange}
		We reiterate for future use the fact arising from the proof that { $G'(x)=G'(x;\alpha,\beta)$ cannot change sign on either of the intervals $(-\infty, z_k)\, , \, (\zeta_\ell,\infty)$, where $z_k=z_k(\alpha,\beta)$ and $\zeta_\ell = \zeta_\ell (\alpha,\beta)$ are as in the proof above.}
	\end{remark}

	Even more, we can derive existence and monotonicity of an additional extremum of the function $g_{\alpha}$ which lies outside the interval $[-1,1]$.
	\begin{proposition}\label{prop:outsideguy}
		Suppose that $\{\pm 1\} \subset \mathcal{A}_n({\alpha})$ and $\alpha$ is neither the tip of a V-shape for $E_{n-1}(\alpha)$ nor part of a V-shape of $E_{n}(\alpha)$.
		Then, there exists { a unique critical point} $w = w(\alpha) \in \R \setminus [-1,1]$ such that $g_{\alpha}'(w)=0$. Moreover, in the intervals of values of $\alpha$ where such $w$ exists, it is also monotonically increasing with respect to $\alpha$.
	\end{proposition}
	\begin{proof} Since $\alpha$ is not a tip of a V-shape for $E_{n-1}(\alpha)$ or $E_{n}(\alpha)$, we know  $p_n(x;\alpha)$ is of degree $n$ and $\mathcal{A}_n(\alpha)  $ has $n+2$ points. 
		We also know that $g_\alpha(x)$ alternates monotonicity on the intervals $(u_{i+1}(\alpha),u_i(\alpha))$, $i=0, \dots, k$, and $(v_{j}(\alpha),v_{j+1}(\alpha))$, $j=0, \dots, \ell$, which implies $g_\alpha'(x)$ changes sign at least $k$ times on $(-1,\alpha)$ and at least $\ell$ times on $(\alpha,1)$ (in particular, $sign\, g'_\alpha(x)=(-1)^k$ on $(-1,u_k(\alpha))$). 
		Consequently { the $(n-2)$-degree polynomial} $g_\alpha ''(x)$ has at least $k-1$ zeros in $(-1,\alpha)$ and at least $\ell-1$ zeros in $(\alpha,1)$. 
		Since $\{ \pm 1\}\subset \mathcal{A}_n({\alpha})$, we have that $ g_{\alpha}(-1) = (-1)^{n+1} g_{\alpha}(+1)$. Thus, there must be { a unique} additional extremum $w=w(\alpha)$ outside $[-1,1]$, because $g_{\alpha}$ behaves as a polynomial of degree $n$ (even or odd) at $\pm \infty$.
		
		\begin{figure}[h]
			\begin{center}
				\includegraphics[scale=0.5]{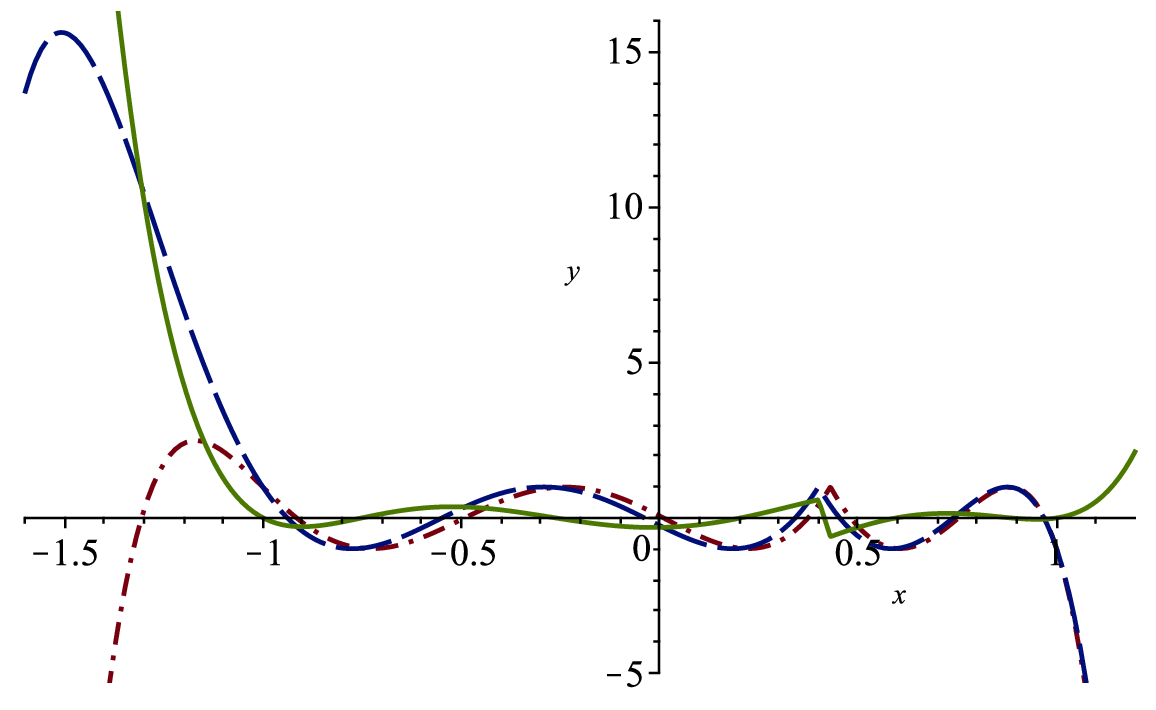}
			\end{center}
			\caption{$\{g_\alpha, g_\beta, G(x)\}$ for $n=6, \alpha=0.4$, and $\beta=0.43$. (Dash, Dash-Dot, and Solid graphs respectively) }
			\label{fig:6'case}
		\end{figure}
		To derive the monotonicity of $w(\alpha)$, let $\beta>\alpha$ be as in the proof of Theorem \ref{thm:monotonicity} with the additional assumption that $w(\alpha)$ and $w(\beta)$ are close to each other and outside $[-1,1]$. As we noted in Remark \ref{nosignchange}, $G'(x)$ preserves sign on $(-\infty,z_k)$ and on $(\zeta_\ell,\infty)$. This implies that $G(x)$ preserves monotonicity on these subintervals.
		
		We first consider the case  $w(\alpha), w(\beta) \in (-\infty,-1)$ (see Figure \ref{fig:6'case} for illustration). Observe that $g_\alpha(x)$ and $g_\beta (x)$ share alternating monotonicity on the intervals $(u_{i+1}(\beta),u_i(\alpha))$, $i=0, \dots, k$.  Recall also that $sign\, G'(z_i)=(-1)^i$, $i=0,\dots,k$, $z_i\in (u_{i+1}(\alpha),u_i (\alpha))$.  As $\alpha$ is not part of a V-shape for $E_n (\alpha)$, we must have $g_\alpha '(-1)\not= 0$. We may assume $\beta>\alpha$ is close enough that $g_\beta '(-1)\not= 0$ as well. Then the functions $G, g_\alpha$, and $g_\beta$  have the same monotonicity in  neighborhoods of $-1$ (indeed, $sign\, G'(z_k)=sign\, g'_\alpha (z_k)=(-1)^k$).
		
		Without loss of generality we assume that the function $G(x)$ is decreasing on $(-\infty,z_k)$, or $G'(x)<0$ in that interval. We have that $g_\alpha (x)$ is decreasing on $(w(\alpha),-1)$ and increasing on $(-\infty,w(\alpha))$. Similarly, we have that $g_\beta(x)$ is decreasing on $(w(\beta),-1)$ and increasing on $(-\infty,w(\beta))$. If we assume that $w(\beta)<w(\alpha)$, as $g_\alpha '(w(\alpha))=0$ one derives $g_\alpha ' (w(\alpha)) - g_\beta ' (w(\alpha))>0$, a contradiction with $G(x)$ being decreasing on $(-\infty,-1)$.
		
		The case  $w(\alpha), w(\beta) \in (1,\infty)$ follows by considering the approximation of the checkmark functions  $f(x;-\alpha)$ and $f(x;-\beta)$ and the symmetry of the minimax problem (see \cite{DLT}).
		
		This proves the proposition.
	\end{proof}
	
	\setcounter{equation}{0}
	\setcounter{theorem}{0}
	\section{Analyticity of $E_n(\alpha)$ and Shekhtman's Conjecture.}
	{
		
		Throughout this section we will show that the minimax error $E_n(\alpha)$ is an analytic function of the parameter $\alpha \in (-1,1)$, except at tips and endpoints of V-shapes (where analyticity actually fails). Furthermore, we will see that $E_n(\alpha)$ has a continuous derivative on any interval excluding tips of V-shapes. The same conclusions will hold for the coefficients of the minimax polynomial $p_n(x;\alpha)$, as well as the alternation points $\{u_i\}$ and $\{v_j\}$. At the end of the section we prove Shekhtman's Conjecture that for odd $n$, $E_n(\alpha)$ has a local maximum at $\alpha=0$.
		
		For values of $\alpha$ in an interval which does not intersect with a V-shape, we are in the situation described just after (\ref{mathcal A}), and we label the alternation points $ -1=u_{k+1}<u_{k}< u_{k-1}<\cdots< u_1< \alpha< v_1< v_2< \cdots< v_\ell<v_{\ell+1}=1$. It is understood that the $u_i$ and $v_j$ depend on $\alpha$, and that $k+\ell=n-1$. Note that for each fixed $\alpha$, $g_\alpha(x)$ is smooth away from $x=\alpha$ and has local extrema at each $x=u_i$ and $x=v_j$. Let us say $p_n(x;\alpha)= c_nx^n + \cdots + c_1 x + c_0$, where the coefficients $c_j$ depend on $\alpha.$
		
		In this case we can phrase the solution to the minimax problem in terms of a system of $2n+1$ equations in $2n+1$ unknowns, with $\alpha$ as an independent parameter. The first $n+2$ equations will describe the equioscillation at the alternation points, and the final $n-1$ equations will state that $g_\alpha'$ vanishes at the $u_i$ and $v_j$. The unknowns are the $c_m$, the $u_i$, the $v_j$ and $E_n$.
		
		{
			\begin{alignat*}{9}\label{impsys}
				\sum_{m=0}^n c_m u_i^m & {}-\alpha & {}+ u_i & {}+ (-1)^{i+1} E_n   & =0 ,& \ \ \ i=k+1,k,\dots, 0; \\
				\sum_{m=0}^n c_m v_j^n & {}+\alpha & {}- v_j & {}+ (-1)^{j+1}E_n   & =0 ,& \ \ \  j=1,\dots,\ell +1;\\
				\sum_{m=1}^n m c_m u_i^{m-1} &   {}+0       &  {}+1   &    {}+0            &=0 ,& \ \ \   i=k,k-1,\dots ,1;\\
				\sum_{m=1}^n m c_m v_j^{m-1} &   {}+0       &  {}-1   &      {}+0              &=0 ,& \ \ \ j=1,\dots, \ell.
			\end{alignat*}
		}

		With an eye toward the analytic implicit function theorem {(see \cite[Theorem 2.3.5]{KP})}, differentiate each of the above equations with respect to each of the unknowns. Gather these partial derivatives in a matrix, where the first $n+1$ columns contain the derivatives with respect to the $c_n, c_{n-1},\cdots, c_0$; column $n+2$ contains the derivatives with respect to $E_n$; and the last $n-1$ columns contain derivatives with respect to the \linebreak $u_k, u_{k-1}, \cdots, u_1, v_1, v_2, \cdots, v_\ell$.  
		
		To be precise, let $\textbf{F} = (F_1,\ldots,F_{2n+1})$ be the vector function $\textbf{F}: \R^{2n+2}\rightarrow \R^{2n+1}$ and $\textbf{y} = (\textbf{x},\alpha) = \left(\{c_{n-m}\}_{m=0}^n, E_n, \{u_{k-i}\}_{i=0}^{k-1}, \{v_j\}_{j=1}^\ell, \alpha \right)  \in \R^{2n+2}$ in such a way that the above system has the vector form
		\begin{equation}\label{sysvector}
			\textbf{F}(\textbf{y}) = \textbf{0}\,,\; \textbf{0}\in \R^{2n+1}\,.
		\end{equation}
		To prove that \eqref{sysvector} implies the existence of an implicit function
		$$\textbf{G}: \R \rightarrow \R^{2n+1}\,,\; \textbf{G}(\alpha) = \textbf{x}\,,$$
		we require that the Jacobian matrix
		\begin{equation}\label{Jacobian}
			\mathcal{J} =\mathcal{J} (\alpha):= \left[\frac{\partial \textbf{F}}{\partial \textbf{x}}\right]
		\end{equation}
		be nonsingular. In that case, we would have
		\begin{equation}\label{derivatives}
			\left[\frac{d \textbf{G}}{d \alpha}\right]_{(2n+1)\times 1} = -\,\mathcal{J}_{(2n+1)\times (2n+1)}^{-1}\,\cdot \,\left[\frac{\partial \textbf{F}}{\partial \alpha}\right]_{(2n+1)\times 1}
		\end{equation}
		The Jacobian matrix $\mathcal{J}$ given in \eqref{Jacobian} may be written in block form:
		\[ \mathcal{J} = \begin{bmatrix}
			A_{(n+2)\times (n+2)} & B_{(n+2)\times(n-1)} \\ C_{(n-1)\times(n+2)} & D_{(n-1) \times (n-1)} \end{bmatrix} \]
		
		First note that $B$ is the null matrix. Indeed, its first and last rows, and all of its off-diagonal entries are $0$ by inspection, and the remaining diagonal entries are $1+p_\alpha'(u_i), i=k,k-1, \dots 1$; and $-1+p_\alpha'(v_j),j=1,2,\dots, \ell$, respectively. Each of these is $0$ according to the final $n-1$ equations in the system.
		
		The matrix $D$ is diagonal by inspection, and its diagonal entries are $p_{\alpha}''(u_i), i=k, k-1, \dots, 1$; and $p_\alpha ''(v_j), j=1,2, \dots, \ell$, respectively. 
		At this point, we conclude that
		\begin{equation}\label{detJ}
			\det \, \mathcal{J}= \det \, A \, \cdot \, \prod_{i=1}^k p''_n(u_i;\alpha) \, \cdot \, \prod_{j=1}^\ell p''_n(v_j;\alpha). \end{equation}	
		Since { $g_\alpha' (x)$ vanishes at $\{w(\alpha), u_k, \cdots, u_1\}$ and $\{v_1, \cdots, v_\ell \}$, we can identify by Rolle's Theorem all $n-2$ roots of $g_\alpha'' (x)=p_n''(x;\alpha)/E_n (\alpha)$ strictly between these points (if $w(\alpha) >1$ then include it after $v_\ell$ instead of before $u_k$)}. Hence, the product of second derivatives in (\ref{detJ}) is nonzero.

		We also see that $\det \, A \neq 0$, where	
		\[A= \begin{bmatrix} (-1)^n		 & (-1)^{n-1}		& 	\cdots 	& -1   	& 1        & (-1)^{k+2} \\
			u_k^n 	& u_k^{n-1}	& \cdots	& u_k 	& 1        & (-1)^{k+1} \\
			\vdots	& \vdots	& \ddots 	& \vdots& \vdots   & \vdots \\
			u_1^n 	& u_1^{n-1}	& \cdots	& u_1 	& 1        & (-1)^{2} \\
			\alpha^n 	& \alpha^{n-1}	& \cdots	& \alpha 	& 1        & (-1)^{1}\\
			v_1^n 	& v_1^{n-1}	& \cdots	& v_1 	& 1        & (-1)^{2} \\
			\vdots	& \vdots	& \ddots 	& \vdots& \vdots   & \vdots \\
			v_\ell^n 	& v_\ell^{n-1}	& \cdots	& v_\ell 	& 1        & (-1)^{\ell+1}\\
			1		 & 1		& 	\cdots 	& 1   	& 1        & (-1)^{\ell+2} \end{bmatrix}.
		\]
		If we let $\hat A$ be the matrix obtained from $A$ by reversing the order of its columns, then the minors of $\hat{A}$ formed along the first column are Vandermonde matrices. We derive
		{ \[\det \, A = (-1)^k\,\cdot (-1)^{(n+2)(n+1)/2} \, \cdot \left(\sum_{i=1}^{k+1}\,{\mathcal V}\, (u_i)\,+\, {\mathcal V}\, (\alpha)\,+\,\sum_{j=1}^{\ell+1}\, {\mathcal V}\, (v_j)\right)\, ,\]}where ${\mathcal V}\, (u_i)\,,\,i=1,\ldots,k+1\,,$ ${\mathcal V}\, (\alpha)$ and ${\mathcal V}\, (v_j)\,,\,j=1,\ldots,\ell+1\,,$ denote the Vandermonde determinants obtained as minors by deleting the first column and the row corresponding to $u_i$, $\alpha$ or $v_j$ respectively in $\hat{A}.$ Since each of these $\mathcal{V}(\cdot)$ is strictly positive, $\det \, A$ is non-zero. 
		
		We thus have the following theorem.
		
		\begin{theorem}\label{thm:diff}
			The error function $E_n(\alpha)$ is analytic in $\alpha$ on $(-1,1)$, with the exception of the endpoints and the tips of V-shapes. The same holds for the alternation points and the coefficients of the minimax polynomial.
		\end{theorem}
		\begin{proof}
			{Suppose $n\geq 2$ (direct computation shows $E_1(\alpha)$ is quadratic). Consider $\alpha$ in an open interval excluding the tips and endpoints of the V-shapes of $E_n(\alpha)$. If the interval is within a V-shape, analyticity follows from the fact that the minimax problem is solved by a linear transformation.} 
			
			For intervals outside of any V-shape, the analytic implicit function theorem can be applied as described just before the statement of the theorem.
		\end{proof}

		We will see later that as $\alpha$ passes into or out of a V-shape, an alternation point may enter the interval $[-1,1]$ from the left, or an alternation point may exit the interval to the right. The details will be explained in Section 4, with the conclusion that at a left endpoint of a V-shape $w(\alpha)=u_{k+1}(\alpha)=-1$ and as $\alpha$ increases $u_{k+1}(\alpha) $ will move into $(-1,1)$ as the leftmost alternation point. At the right endpoint of a V-shape, $v_\ell$ will reach $x=1$, and then $v_{\ell-1}$ will become the rightmost interior alternation point. (Consequently, $v_\ell$ will thereafter fail to exist according to our numbering convention.)
		
		With this in mind, we can elaborate on the previous theorem. Namely, on any interval excluding tips of V-shapes, we will conclude $\mathcal{C}^1$ smoothness for $E_n(\alpha)$, the $c_m$, and any of the $u_i$ and $v_j$ which exist on the whole interval.
		
		Given the theorem above, we need only show that the derivatives of these quantities are continuous at the endpoints of V-shapes. Let $\gamma$ be the left endpoint of a V-shape (the case of a right endpoint can be treated similarly.) For $\alpha<\gamma$ and close to $\gamma$, let the minimax problem be described exactly as laid out above in \eqref{sysvector}-\eqref{detJ}.  
		
		For $\alpha > \gamma$, $x=-1$ is no longer an alternation point, and instead an extra interior alternation point $u_{k+1}<u_k$ appears.  We can still describe the minimax problem as a system of equations, with the following modifications. In the system above we used $u_{k+1} := -1$; but now allow $u_{k+1}$ to be another unknown in its own right. We also append an extra equation to the system: $g_{\alpha}'(u_{k+1})=0$. Let this modified system be denoted $\widetilde{\bf{F}}(\widetilde{\bf{y}},\alpha): \mathbb{R}^{2n+3}\to\mathbb{R}^{2n+2}.$  Correspondingly, we let the Jacobian of the new system be formed by placing derivatives with respect to $u_{k+1}$ into column $2n+2$, and call the matrix $\widetilde{\mathcal{J}}(\alpha)$, namely
		\begin{equation}\label{Jtilde}
			\widetilde{\mathcal{J}}(\alpha):=
			\begin{bmatrix}
				\mathcal{J}(\alpha)&\mathbf{0}\\
				\mathbf{r}(u_{k+1})&p_n''(u_{k+1})
			\end{bmatrix},
		\end{equation}
		where $\mathbf{r}(u_{k+1})=[nu_{k+1}^{n-1},\dots,1,0,\dots,0]$ is a $2n+1$-dimensional row vector and $\mathbf{0}$ is the $(2n+1)$-dimensional column zero vector.  Let $\widetilde{A}(\alpha)$ be the same as matrix $A(\alpha)$, except with $u_{k+1}$ replacing $-1$ in the relevant entries of the first row. From Section 4, $\widetilde{A}(\gamma)={A}(\gamma)$, since at that point $u_{k+1}=-1$.  
		
		In \eqref{derivatives}, note that
		$$\left[\frac{\partial \textbf{F}}{\partial \alpha}\right]^T\,=\,\begin{bmatrix} -1 & \cdots & -1 & p_n'(\alpha;\alpha) & 1 & \cdots & 1 &0 & \cdots & 0 \end{bmatrix},$$
		where $-1$, $1$ and $0$ are repeated $k+1$, $\ell+1$ and $n-1$ times, respectively. The matrix
		$\left[\frac{\partial\widetilde{ \textbf{F}}}{\partial \alpha}\right]^T$
		is of exactly the same form, except with one extra $0$ entry on the right.
		
		For notational convenience, let
		
		\[\bf{F}'(\gamma):=\lim_{\alpha \to \gamma^-} \left[\frac{\partial \textbf{F}}{\partial \alpha}\right]= \begin{bmatrix} -1 & \cdots & -1 &p_n'(\gamma; \gamma) & 1 \cdots & 1 &0 & \cdots & 0 \end{bmatrix} \]
		and similarly let
		\[ \widetilde{\bf{F}}'(\gamma):=\lim_{\alpha \to \gamma^+} \left[\frac{\partial \widetilde{\textbf{F}}}{\partial \alpha}\right]= \begin{bmatrix} -1 & \cdots & -1 &p_n'(\gamma; \gamma) & 1 \cdots & 1 &0 & \cdots & 0 \end{bmatrix}, \]
		where the vector $\widetilde{\bf{F}}'(\gamma)$  has an extra 0 entry on the right in comparison to $\bf{F}'(\gamma)$. In Section $4$ we will see that as $\alpha \to \gamma$, the external extremum $w(\alpha)$ of $g_\alpha(x)$ reaches $x=-1$ to coincide with (and in fact become) $u_{k+1}$.  We have seen already that the roots of $p_n''(x; \alpha)$ are strictly between the critical points of $g_\alpha(x)$, and so $p_n'' (u_{k+1};\alpha)\not= 0$. From \eqref{Jtilde} we have $$\det \widetilde{\mathcal{J}}(\alpha)=p_n'' (u_{k+1};\alpha) \det \mathcal{J}(\alpha)\not= 0,$$ so $  \widetilde{\mathcal{J}}(\alpha)$ is invertible as well. Direct inspection reveals that
		\[
		\widetilde{\mathcal{J}}(\alpha)^{-1}=
		\begin{bmatrix}
			\mathcal{J}(\alpha)^{-1}&\mathbf{0}\\
			-p_n''(u_{k+1})^{-1}\mathbf{r}(u_{k+1})\mathcal{J}(\alpha)^{-1}&p_n''(u_{k+1})^{-1}
		\end{bmatrix},
		\]
		Moreover, $\mathcal{J}(\alpha)$ and $\widetilde{\mathcal{J}}(\alpha)$ remain invertible even for $\alpha=\gamma$. 
		
		At this point, we have
		
		\[ \bf{G}'(\gamma^-) := \lim_{\alpha \to \gamma^-} \bf{G}'(\alpha)= - \mathcal{J}^{-1}(\gamma) \cdot \bf{F}'(\gamma),  \]
		
		\[\widetilde{\bf{G}} '(\gamma^+):= \lim_{\alpha \to \gamma^+} \widetilde{\bf{G}}'(\alpha)=- \widetilde{\mathcal{J}}^{-1}(\gamma) \cdot \widetilde{\bf{F}}'(\gamma) .\]
		Since $\lim_{\alpha \to \gamma^-}A(\alpha)=\lim_{\alpha \to \gamma^+} \widetilde{A}(\alpha)$, the entries of $\bf{G}'(\gamma^-)$ are identical to the first $2n+1$ entries of $\widetilde{\bf{G}}'(\gamma)$. 
		
		Hence $E_n(\alpha)$, the $c_m$, the $u_i$ and the $v_j$ for $i=1,2,\dots, k; j=1,2, \dots,\ell$ have continuous derivative in $\alpha$, even at the point $\alpha=\gamma$. We state this as a theorem

		\begin{theorem} On any interval of $\alpha$ values excluding tips of V-shapes, $E_n(\alpha)$ has a continuous derivative with respect to $\alpha$, as do the minimax coefficients $\{c_m\}$ and the alternation points $\{u_i\}, \{v_j\}$ whenever they exist. \end{theorem}
		
		As an immediate corollary, we can strengthen Theorem \ref{thm:monotonicity} to conclude strict monotonicity of the alternation points $u_i$, $v_j$ with respect to $\alpha$.
		
		\begin{corollary} \label{strict} On any interval of $\alpha$ values excluding tips and endpoints of V-shapes, the alternation points $u_i$, $v_j$, $i=1,2 \dots, k; j=1,2, \dots,\ell$ are strictly increasing in $\alpha$. \end{corollary}
		\begin{proof} By analyticity, if any $u_i$ were to be constant on an interval of $\alpha$ values, then this $u_i$ would be constant for all $\alpha$ up to the next (or previous) endpoint $\gamma$ of a V-shape. But we know by linear transformation that within the V-shape, $u_i'(\alpha)$ will be a non-zero constant. This contradicts the continuity of $u_i'(\alpha)$ at $\alpha=\gamma$. Similarly for the $v_j$.
		\end{proof}
		
		We now turn to Shekhtman's conjecture, to establish that $E_n(\alpha)$ has a local maximum at $\alpha=0$ for $n$ odd.  As above, the Implicit Function Theorem also allows us to find an expression for $E_n'(\alpha)$ in terms of the alternation points. From (\ref{derivatives}) we have
		\begin{equation}\label{deriverror}
			{E'_n(\alpha) = \,-\,\mathcal{R}_{n+2}\,\left[\frac{\partial \textbf{F}}{\partial \alpha}\right]\,,}
		\end{equation}
		where $\mathcal{R}_{n+2}$ stands for the $(n+2)$--{row} of matrix $\mathcal{J}^{-1}$.

		By Cramer's Rule, this can be written $E_n'(\alpha)=\det \, A_{n+2} \, / \det \,A,$ where
		$A_{n+2}$ is the matrix obtained from $A$ by replacing the final column with 
		
		\[\begin{bmatrix} 1 & 1 & \cdots &1& -p_n'(\alpha; \alpha) & -1 \cdots & -1 \end{bmatrix}^T. \]

		Denoting by $\hat{A}_{n+2}$ the matrix obtained from $A_{n+2}$ by reversing the order of the columns, we can represent 
		$E_n'(\alpha)$ as
		\[E_n'(\alpha)= \frac{ \det \, {\hat{A}_{n+2}}}{\det \, \hat{A}}. \]
		
		But notice now that by taking an appropriate linear combination of columns $2$ through $n+1$ in $\hat{A}_{n+2}$, we can produce values of the derivative of the minimax polynomial. Adding this linear combination of columns to the first column will not change the determinant of the matrix, but the new first column will be
		\[ \begin{bmatrix} (p_n'(-1;\alpha)+1) & 0 & 0 & \cdots &0 & (p_n'(1;\alpha)-1) \end{bmatrix}^T. \]
		Then by expansion on this column,
		\[\det \, \hat{A}_{n+2} = (p_n'(-1; \alpha)+1 ) \cdot \mathcal{V}(-1)+ (-1)^{n+3} \, (p_n'(1;\alpha)-1) \cdot \mathcal{V}(1) \]
		So at this point we have the following formula for the derivative of the error function: \begin{equation}
			\label{errorprime}
			E_n'(\alpha)= \frac{(p_n'(-1; \alpha)+1 ) \, \cdot \, \mathcal{V}(-1)+ (-1)^{n+1} (p_n'(1;\alpha)-1)\, \cdot \, \mathcal{V}(1)}{(-1)^k\,\left(\sum_{i=1}^{k+1}\,{\mathcal V}\, (u_i)\,+\, {\mathcal V}\, (\alpha)\,+\,\sum_{j=1}^{\ell+1}\, {\mathcal V}\, (v_j)\right)}
		\end{equation} 
		
		Observe that $\mathcal{V}(1)$ and $\mathcal{V}(-1)$ share as a common factor the product of mutual distances among the $u_i,\alpha, v_j$. This itself is the Vandermonde determinant of $u_k, \cdots, u_1, \alpha, v_1, \cdots, v_\ell$, which we denote by $\mu(\alpha)$.
		
		The remaining factors of $\mathcal{V}(1)$ and $\mathcal{V}(-1)$ are the products of the distances of each alternation point to either $-1$ or $1$ respectively. We give names to these as well:
		\[\delta_{-1}(\alpha) := (\alpha+1) \cdot \prod_{i=1}^k (u_i+1) \cdot \prod_{j=1}^\ell (v_j+1), \]
		\[\delta_1 (\alpha) := (1-\alpha) \cdot \prod_{i=1}^k (1-u_i) \cdot \prod_{j=1}^\ell (1-v_j). \]
		
		Thus we have
		\[ \mathcal{V}(-1)= \mu(\alpha) \cdot \delta_{1}(\alpha),\]
		\[\mathcal{V}(1)= \mu(\alpha) \cdot \delta_{-1}(\alpha). \]
		
		For convenience, we also let $\Delta$ denote the quantity
		\[\Delta(\alpha)= \frac{\mu(\alpha)}{\sum_{i=1}^{k+1}\,{\mathcal V}\, (u_i)\,+\, {\mathcal V}\, (\alpha)\,+\,\sum_{j=1}^{\ell+1}\, {\mathcal V}\, (v_j)}.\]
		
		In order to prove Shekhtman's conjecture, it will be advantageous to divide (\ref{errorprime}) through by $E_n(\alpha)$, to get the following expression for the logarithmic derivative of the error function: \begin{equation}
			\label{logderiv}
			\frac{E_n'(\alpha)}{E_n(\alpha)} = (-1)^k \cdot \Delta (\alpha) \cdot \big{[} \, g_{\alpha} '(-1) \, \delta_1(\alpha) + (-1)^{n+3} \, g_{\alpha}'(1) \, \delta_{-1}(\alpha) \big{]}
		\end{equation} With this we are ready to prove the conjecture.
		
		\begin{theorem}
			\label{Shekhtman}
			
			For each odd $n$, the function $E_n(\alpha)$ has a local maximum at $\alpha=0$.
		\end{theorem}
		
		\begin{proof} For $n$ odd, (\ref{logderiv}) takes the form 
			\[\frac{E_n'(\alpha)}{E_n(\alpha)} = (-1)^k \cdot \Delta (\alpha) \cdot \big{[} \, g_{\alpha} '(-1) \, \delta_1(\alpha) + g_{\alpha}'(1) \, \delta_{-1}(\alpha) \big{]}. \] 
			
			As Lemma \ref{tips} below will indicate, since $\alpha=0$ is a tip of a V-shape for $E_{n-1}(\alpha)$ (cf. \cite{DLT}), it follows that $\alpha=0$ is not within a V-shape of $E_{n}(\alpha)$. If $\alpha=0$ were the endpoint of a V-shape for $E_n(\alpha)$, then by symmetry it would be a tip of a V-shape, which was just excluded. So by Theorem \ref{thm:diff}, $E_n(\alpha)$ is analytic at $\alpha=0$ and the use of (\ref{logderiv}) is justified.
			
			We note that $\Delta(\alpha)$ is always positive, and by symmetry $E_n'(0)=0,$ $g_0'(-1)=-g_0'(1) \neq 0$, and $\delta_1(0)=\delta_{-1}(0).$ Furthermore, by symmetry at $\alpha=0$ we have $k=\ell$, and this persists for $\alpha$ sufficiently close to $0$.
			
			Now we use the first derivative test. We have just mentioned that $E_n'(0)=0$, so if we can establish that $E_n'(\alpha)<0$ for small positive $\alpha$, then the symmetry of the problem will ensure $E_n'(\alpha)>0$ for small $\alpha<0$, and then the critical point $\alpha=0$ will necessarily be a local maximum of $E_n(\alpha)$. So we only need to establish that for small $\alpha>0$, we have 
			\[(-1)^k \, \big{[} \, g_{\alpha} '(-1) \, \delta_1(\alpha) + g_{\alpha}'(1) \, \delta_{-1}(\alpha) \big{]} <0. \]

			To demonstrate this, for a fixed $\alpha>0$ close to $0$, consider the function $G(x)=g_{0}(x)-g_{\alpha}(x)$ as in (\ref{functiong}). In the proof of Theorem \ref{thm:monotonicity}, we demonstrated the existence of points $z_i \in (u_{i+1}(0),u_{i}(0))$, $i=0, 1, \cdots, k$ where $sign \, G'(z_i)=(-1)^i$, and points $\zeta_j \in (v_{j}(\alpha),v_{j+1}(\alpha)),$ $j=0,1,\cdots, \ell$ where $sign \, G'(\zeta_j)=(-1)^j.$ These in turn led to the existence of points $y_i \in (z_{j+1}, z_j)$ and $\eta_j \in (\zeta_j, \zeta_{j+1})$ where $sign \, G''(y_i)=(-1)^{i-1}$ and $sign \, G''(\eta_j)=(-1)^j$. We saw also that all the roots of $G''(x)$ are in the interval $(y_k, \eta_\ell).$
			
			We have that $sign \, G'(-1)= sign \, G'(z_k)= (-1)^k$ (see Remark \ref{nosignchange}).  
			
			In other words
			\[sign \, (-1)^k(g_0'(-1)-g_\alpha'(-1))= 1,\]
			
			or 
			\[(-1)^kg_0'(-1)>(-1)^kg_{\alpha}'(-1).\]
			And since $k=\ell$, we similarly get 
			\[(-1)^kg_0'(1)>(-1)^kg_\alpha '(1).\]
			
			By continuity
			\[sign \, g_0'(-1)= sign \, g_\alpha '(-1) = (-1)^k\]
			
			and
			\[sign \, g_0 '(1) = sign \, g_\alpha '(1) = (-1)^{k+1}.\]
			
			At this point we have
			\[ | \, g_\alpha '(-1) \, |\,= \, (-1)^k \, g_\alpha '(-1) \; < \; (-1)^k \, g_0'(-1)\, = \,|\, g_0'(-1)|\, \]
			as well as	
			\[ -|\,g_\alpha '(1)\,|\,=\, (-1)^k \, g_\alpha ' (1)\;< \; (-1)^k \, g_0'(1)\, = - |\, g_0'(1)\, |.  \]	
			We have also seen by symmetry that $|g_0'(-1)|=|g_0'(1)|,$ and the monotonicity of alternation points in Corollary \ref{strict} makes it clear that $\delta_1(\alpha)<\delta_1(0)=\delta_{-1}(0)<\delta_{-1}(\alpha).$ So finally, combining this all together we have
			\[(-1)^k \, \big{[} \, g_{\alpha} '(-1) \, \delta_1(\alpha) + g_{\alpha}'(1) \, \delta_{-1}(\alpha) \big{]} \, < \, |\,g'_0(-1)\,| \, \cdot \, \delta_1(\alpha) - |\,g'_0(1)\,| \, \cdot \, \delta_{-1}(\alpha), \]
			and the right hand side of this inequality is clearly negative.

		\end{proof}
		
		Having established Shekhtman's conjecture, we formulate the following complementary one.
		
		{\bf Conjecture.} For all natural numbers $n$, the function $E_n(\alpha)$ is concave outside of its V-shapes. For odd $n$, $E_n(\alpha)$ has an absolute maximum at $\alpha=0$.
		
	}
	\section{The V-shapes and the Phase diagram}
	
	We now turn to proving the conjecture of \cite{DLT} that $n-1$ is the exact number of V-shapes in the graph of $E_{n}(\alpha)$. This will follow from a string of lemmas based on studying the dynamics of the external extremum $w(\alpha)=w(\alpha;n)$ of $g_\alpha(x)=g_{\alpha,n}(x)$. { Since degrees $n-1$ and $n$ shall appear together, the explicit dependence on degree is added for clarity.}
	
	\begin{lemma} \label{interlace} A tip of a V-shape of $E_{n-1}(\alpha)$ cannot be within a V-shape of $E_n(\alpha)$.
		\label{tips}
		\begin{proof}  
			
			At the tip of a V-shape of $E_{n-1}(\alpha)$, $( \,p_{n-1}(x;\alpha)-|x-\alpha| \,)$ will have $(n-1)+3-3=n-1$ turning points in the interval $[-1,1].$ This is because in case only $n-2$ turning points were in the interval, no linear transformation of the domain would preserve the minimax property, in contradiction to being at the tip of a V-shape. Similarly in a V-shape of $E_n(\alpha)$, there must be $n$ turning points for $(\,p_n(x;\alpha)-|x-\alpha| \,)$ in the interval $[-1,1]$. 
			
		\end{proof}
		
	\end{lemma}

	\begin{lemma}
		Let $\alpha_0$ be a tip of a V-shape from $E_{n-1}(\alpha)$. Let $w(\alpha;n)$ be the extremum of $g_{\alpha,n}(x) $ which lies outside $[-1,1]$, when such exists. Then  $\lim_{ \alpha \to \alpha_0^+} w(\alpha;n)=-\infty$ and $\lim_{\alpha \to \alpha_0^-} w(\alpha;n)=+\infty$.
		
	\end{lemma}
	
	\begin{proof}
		By Lemma \ref{interlace} $\alpha_0$ is not within a V-shape of $E_n(\alpha)$.Therefore, we know that $w(\alpha;n)$ exists for all $\alpha$ in a punctured neighborhood of $\alpha_0$, and that it is monotone increasing in $\alpha$ on either side of $\alpha_0$ .  
		
		Assume now that $w(\alpha;n)>1$, and let $q_{\alpha,n}(x):=\frac{p_n'(x;\alpha)-1}{E_n(\alpha)}$. For $x>\alpha$ we have $g_{\alpha,n}'(x)=q_{\alpha,n} (x),$ and in particular $q_{\alpha,n}(w(\alpha;n))=0$.
		
		Since the coefficients of the minimax polynomial along with the minimax error are analytic in $\alpha$ in a neighborhood of $\alpha_0$, we know that for some analytic functions $b_j(\alpha;n)$, $j=0,1,\cdots, n-1$, we have for all $x$: \begin{equation}\label{qAnaliticity}
			q_{\alpha,n} (x)= \sum_{j=0}^{n-1}b_j(\alpha;n)x^j.  
		\end{equation}
		
		In the proof of Proposition \ref{prop:outsideguy} we established that $g_{\alpha,n}'' (x)$, a polynomial of degree $n-2$, has $n-3$ roots in $(-1,1)$. Since $g_{\alpha,n}$ has an extremum at $w(\alpha;n)$, if $w(\alpha;n) $ is a root of $ g_{\alpha,n}'' (x)$, it has to be of multiplicity at least two, which is a contradiction. Therefore, $g_{\alpha,n} '' (w(\alpha;n))=q_{\alpha,n}  ' (w(\alpha;n))\not= 0$. Together with \eqref{qAnaliticity}, this implies that $w(\alpha;n)$ is analytic in $\alpha$ when it exists.  
		
		As $w(\alpha;n)$ is monotone, we also know that $\omega=\lim_{\alpha \to \alpha_0^-}w(\alpha;n)$ exists. Assume for contradiction that $\omega$ were finite. In that case, we should have:
		\[q_{\alpha_0,n}(\omega)=\sum_{j=1}^{n-1}b_j(\alpha_0;n)\omega^j=\lim_{\alpha \to \alpha_0^-} \sum_{j=1}^{n-1} b_j(\alpha;n) w(\alpha;n)^j \]
		\[= \lim_{\alpha \to \alpha_0^-} q_{\alpha,n}(w(\alpha;n))=\lim_{\alpha \to \alpha_0^-} 0 = 0. \]
		
		Hence $g_{\alpha_0,n}$ has a critical point outside $[-1,1]$, which is impossible at the tip of a V-shape (recall that in this case $g_{\alpha_0,n}\equiv g_{\alpha_0,n-1}$). So we must have $\lim_{\alpha \to \alpha_0^-}w(\alpha;n)=+\infty$.  The remaining cases are handled similarly.
		
	\end{proof}
	
	\begin{lemma}
		\label{endpoint}
		Given a V-shape of $E_n(\alpha)$ with left endpoint $\gamma$ and right endpoint $\beta$, we have
		\[\lim_{\alpha \to \beta^+} w(\alpha;n)=1, \]
		
		and 
		\[\lim_{\alpha \to \gamma^-} w(\alpha;n)=-1.  \]
	\end{lemma}
	
	\begin{proof}
		Note that $g_{\beta,n}(x)$ must have a local extremum at $x=1.$ For if $|g_{\beta,n}(1)|<1$, then by either shrinking or expanding the domain slightly by a linear transformation we could keep the needed number of alternation points within the interval $[-1,1].$ Then $\beta$ would be an interior point of a linear part of a V-shape, instead of an endpoint.
		
		Thus, without loss of generality, assume $g_{\beta,n} (1)=1$. Observe that there is $\epsilon >0$ such that for $\alpha \in (\beta-\epsilon, \beta)$ we have $g_{\alpha,n} ' (x)=q_{\alpha,n} (x)<0$ for $x\in [1,(\beta+1+\epsilon)/(\beta+1-\epsilon)]$. Continuity of the coefficients $b_j (\alpha;n)$ in \eqref{qAnaliticity} implies $g_{\beta,n} '(x)=q_{\beta,n} (x)\leq 0$ in this subinterval. Therefore, we conclude that $g_{\beta,n}'(1)=0$ and $x=1$ is a local maximum for $g_{\beta,n}(x)$.

		Now since $w(\alpha,n)$ is monotone in the interval $(\beta, \beta +\epsilon)$ for any small $\epsilon$, we have that $\lim_{ \alpha \to \beta^+} w(\alpha,n)$ exists. By continuity of the minimax polynomial coefficients, and since $g_{\alpha,n}$ has only one external extremum for such $\alpha$, we must conclude that $\lim_{\alpha \to \beta^+}w(\alpha)=1$, to match the location of the extremum of $g_{\beta,n}$ at $x=1$.
		
		The other limit in the statement of the Lemma may be analyzed similarly.	
	\end{proof}
	
	Next we derive an {\em interlacing property} of the V-shapes of $E_{n-1}(\alpha)$ and $E_n (\alpha)$.
	
	\begin{lemma}\label{L4.6}
		The following hold:
		
		\begin{itemize}
			
			\item Between any two consecutive tips of V-shapes of $E_{n}(\alpha)$, there is a tip of a V-shape of $E_{n-1}(\alpha)$.
			
			\item  Between any two consecutive tips of V-shapes of $E_{n-1}(\alpha)$, there is a tip of a V-shape of $E_{n}(\alpha).$
			
			\item Between the rightmost V-shape of $E_{n-1}(\alpha)$ and $\alpha=1$, there is a V-shape of $E_n(\alpha)$.
			
			\item Between $\alpha=-1$ and the leftmost V-shape of $E_{n-1}(\alpha)$, there is a V-shape of $E_n(\alpha).$  
			
		\end{itemize}
		
	\end{lemma}
	
	\begin{proof}
		
		The proof hinges on tracking the progress of $w(\alpha,n)$, the external extremum of $g_{\alpha,n}$, as $\alpha$ increases.
		
		For the first statement, let $\alpha_0$ and $\alpha_1$ be consecutive tips of V-shapes of $E_{n}(\alpha)$.  As we let $\alpha$ grow from $\alpha_0$, it will eventually exit the V-shape of $\alpha_0$, and then we will observe $w(\alpha;n)$ growing monotonically from a value of $1$. But as $\alpha$ approaches $\alpha_1$, it will enter the V-shape of $\alpha_1$ from the left side, and so we observe $w(\alpha;n)$ growing monotonically to a value of $-1$. Hence, there must have been a discontinuity of $w(\alpha;n)$ somewhere in between.  But $w(\alpha;n)$ is analytic wherever it exists, and it exists everywhere outside the V-shapes of $E_n(\alpha)$, except at tips of V-shapes of $E_{n-1}(\alpha).$ So, we must conclude that there occurred a tip of a V-shape of $E_{n-1}(\alpha)$, where $w(\alpha;n)$ jumped from $+\infty$ to $-\infty$.
		
		For the second statement, let $\alpha_0$ and $\alpha_1$ now be two consecutive tips of V-shapes of $E_{n-1}(\alpha).$  As $\alpha$ grows from $\alpha_0$ to $\alpha_1$, we see from Lemma \ref{endpoint} that $w(\alpha;n)$ starts from $-\infty$ at $\alpha_0$, and tends to $+\infty$ at $\alpha_1$. Somewhere in between, we see by continuity that $w(\alpha;n)=-1$, and we enter a V-shape of $E_n(\alpha) $.

		The third statement follows directly from \cite[Theorem 2]{DLT} where the location of the rightmost V-shape of $E_n(\alpha)$ is determined explicitly and the fourth statement follows by the even symmetry of the problem.

	\end{proof}

	Now we may finally establish the number of V-shapes of $E_n$.
	
	\begin{theorem}
		\label{thm:vshapes}
		For each $n \in \mathbb{N}$, the number of V-shapes in the graph of $E_n(\alpha)$ is exactly $n-1$.
	\end{theorem}
	
	\begin{proof}
		
		For each $j \in \mathbb{N}$, define $\mathcal{N}(j)$ to be the number of V-shapes in the graph of $E_j(\alpha)$.
		The case $\mathcal{N}(1)=0$ is immediate since $E_1(\alpha)$ is quadratic, and the cases $\mathcal{N}(2)=1$ and $\mathcal{N}(3)=2$ are established in \cite{DLT}.
		
		Now consider an integer $n \geq 4$. For induction, assume that $\mathcal{N}(n-1)=n-2,$ and let the corresponding tips of V-shapes of $E_{n-1}(\alpha)$ be called $\alpha_1 < \alpha_2 < \cdots < \alpha_{n-2}.$ By Lemma \ref{L4.6} there is at least one V-shape of $E_n(\alpha)$ in each of the intervals $(-1, \alpha_1), (\alpha_1, \alpha_2), \cdots, (\alpha_{n-2}, 1)$. Hence $\mathcal{N}(n) \geq n-1$.
		
		Likewise, between any two V-shapes of $E_n(\alpha)$ must occur at least one of the $\alpha_k$.  By the pigeon-hole principle $\mathcal{N}(n) \leq n-1$, which finishes the induction.
	\end{proof}
	
	{
		To finish, we use Theorem \ref{thm:monotonicity} and Proposition \ref{prop:outsideguy} to completely describe the phase transitions of the alternation points as $\alpha$ increases from $-1$ to $+1$. 
		
		\begin{figure}
			\begin{center}
				\includegraphics[scale=.4]{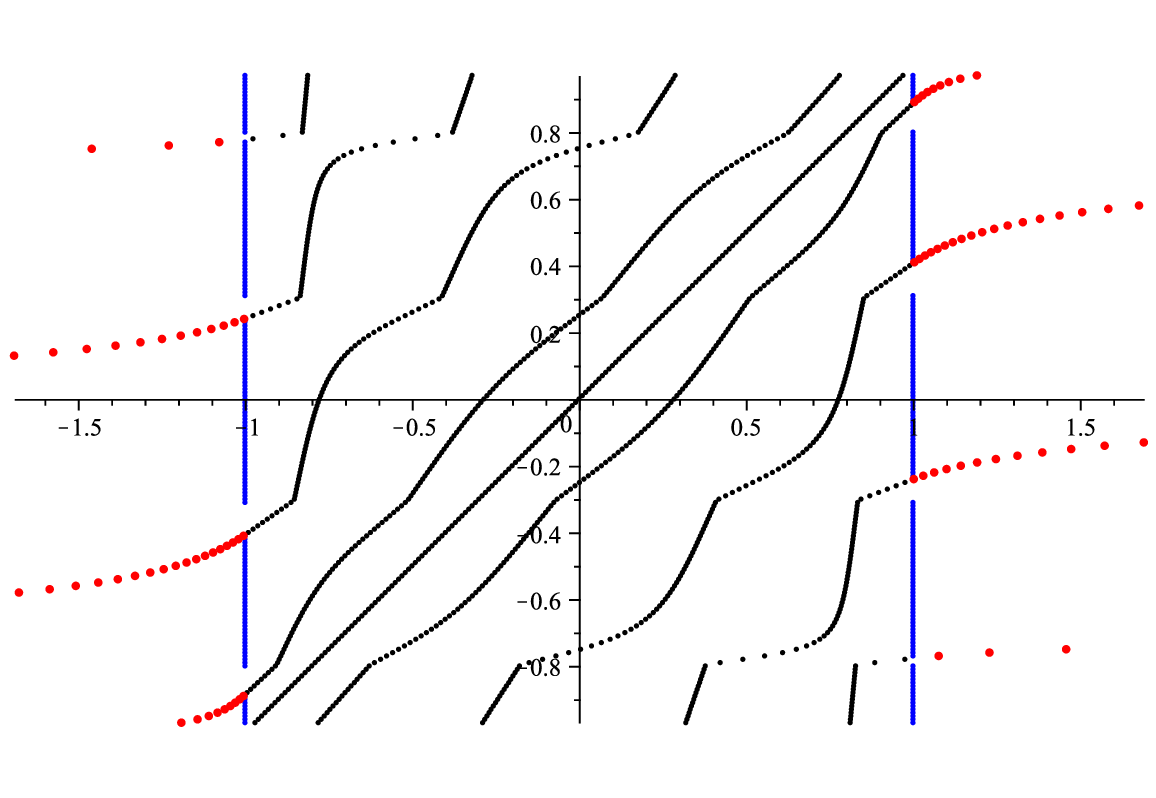}
			\end{center}
			\caption{Phase transitions of the alternation points and the external extremum $w(\alpha)$ (red) in terms of $\alpha$. Values of $\alpha$ are plotted vertically from -1 to +1; at each value of $\alpha$ the locations of the corresponding alternation points are marked horizontally together with $w(\alpha)$ when it exists.}
			\label{fig:phase}
		\end{figure}
		
		The following pattern repeats through each of the V-shapes of $E_n(\alpha)$. Throughout a V-shape, an external extremum $w(\alpha)$ will cross $x=-1$, and replace $-1$ as the leftmost alternation point. Then at the tip of the V-shape, $-1$ will again become an alternation point, making for $n+3$ alternation points. Then $x=+1$ loses its status as an alternation point, until at the end of the V-shape, $x=+1$ will again become an alternation point as an external extremum $w(\alpha)$ emerges from $x=+1$ and moves to the right. After this process, the rightmost $v_\ell$ will have exited the interval and fail to exist, and we will have gained one more $u_{k+1}$. En route to the next V-shape, there will come a point where the previous degree error $E_{n-1}(\alpha)$ has a tip of its V-shape intersecting the graph of $E_n(\alpha)$. At this point, the external extremum $w(\alpha)$ increases without bound and jumps from $+\infty$ to $-\infty$. Then $w(\alpha)$ moves in toward $x=-1$ to begin another phase transition.  
		
		These transitions for $n=5$ are pictured in Figure \ref{fig:phase}. The values of $\alpha$ are plotted vertically from $\alpha=-1$ to $\alpha=+1$ in increments of $0.01$. At each value of $\alpha$, the positions of the alternation points and of $w(\alpha)$ are marked horizontally between $x=-1.7$ and $x=+1.7$. (For reference, the diagonal line through the origin is the alternation point $\alpha$ itself and the vertical bars appear whenever $\pm 1$ are alternation points.)}
	
	\
	
	{\bf Acknowledgment:} We would like to extend our sincere thanks to the anonymous referees, whose thoughtful comments and reviews allowed us to enhance the exposition, as well as to V. Totik and A. Aptekarev for bringing to our attention the works in Russian of Bernstein \cite{Bernstein2} and Nikolsky \cite{N}.
	
	\bibliography{allDLO}

\begin{thebibliography}{10}

\bibitem{Bernstein1}
S.~N. Bernstein.
\newblock Sur le meilleure approximation de $|x|$ par des polynomes de
  degr{\'e}s donn{\'e}s.
\newblock {\em Acta Math.}, 37:1--57, 1913.

\bibitem{Bernstein2}
S.~N. Bernstein.
\newblock On the best approximation of $|x-c|^p$.
\newblock {\em Dokl. Acad. Nauk SSSR}, 18:374--384, 1938.

\bibitem{deVore}
R.~A. DeVore and G.G. Lorentz.
\newblock {\em Constructive Approximation}.
\newblock Number 303 in Grundlehren der Mathematischen Wissenschaften.
  Springer-Verlag, Berlin, 1993.

\bibitem{DLT}
P.D. Dragnev, D.A. Legg, and D.W. Townsend.
\newblock Polynomial approximation of the checkmark function.
\newblock In M.~Neamtu and E.~B. Saff, editors, {\em Advances in Constructive
  Approximation Vanderbilt 2003}, pages 139--164. Nashboro Press, 2004.

\bibitem{KP}
S.G. Krantz and H.R. Parks.
\newblock {\em A primer of real analytic functions}.
\newblock Birkh{\"a}user, Boston, 2002.

\bibitem{N}
S.~Nikolsky.
\newblock On the best approximation in the mean to the function $|a-x|^s$ by
  polynomials ({R}ussian).
\newblock {\em Izvestia Akad. Nauk SSSR}, 11:139--180, 1947.

\bibitem{PePo}
P.P. Petrushev and V.A. Popov.
\newblock {\em Rational approximation of real functions}.
\newblock Number~28 in Encyclopedia of Mathematics and its Applications.
  Cambridge University Press, Cambridge, 2002.

\bibitem{Remez}
E.~Ya. Remez.
\newblock Sur le calcul effectif des polynomes d'approximation de
  {T}schebyscheff.
\newblock {\em C. R. Acad. Sci. Paris}, 199:337--340, 1934.

\bibitem{Shekt}
B.~Shekhtman.
\newblock Private Communication.

\bibitem{Stahl}
H.~Stahl.
\newblock Poles and zeros of best rational approximants of $|x|$.
\newblock {\em Constr. Approx.}, 10(4):469--522, 1994.

\bibitem{T}
V.~Totik.
\newblock Approximation on compact subsets of ${R}$.
\newblock In M.D. Buhmann and D.H. Mache, editors, {\em Advanced Problems in
  Constructive Approximation}, pages 263--274. Birkh{\"a}user Verlag, 2002.

\end{thebibliography}
	\bibliographystyle{}

\end{document}